\documentclass{amsart}

 \usepackage{graphicx}

\newtheorem{theorem}{Theorem}[section]
\newtheorem{lemma}[theorem]{Lemma}

\theoremstyle{definition}

\newtheorem{corollary}[theorem]{Corollary}

\newcommand{\R}{{\Bbb R}}
\newcommand{\Z}{{\Bbb Z}}
\newcommand{\N}{{\Bbb N}}

\newcommand{\T}{{\Bbb T}}
\newcommand{\Q}{{\Bbb Q}}

\theoremstyle{remark}
\newtheorem{remark}[theorem]{Remark}

\numberwithin{equation}{section}

\newcommand\mL{L\kern-0.08cm\char39}

\begin{document}


\title[Minimal non-invertible 
maps on 2-manifolds]{Construction of minimal non-invertible skew-product maps on 2-manifolds}

\author{Jakub \v Sotola}
\address{Mathematical Institute, Silesian University in Opava, Na Rybn\'i\v cku 1, 746 01, Opava, Czech Republic}
\email{Jakub.Sotola@math.slu.cz}

\author{Sergei Trofimchuk}
\address{Mathematical Institute, Silesian University in Opava, Na Rybn\'i\v cku 1, 746 01, Opava, Czech Republic}
\curraddr{Instituto de Matem\'atica y Fisica, Universidad de Talca, Casilla 747,
Talca, Chile}
\email{trofimch@inst-mat.utalca.cl}


\date{July 17, 2014} 


\commby{}

\begin{abstract}   Applying the Hric-J\"ager blow up technique, we give an affirmative answer to the question about the existence of  non-invertible minimal circle-fibered self-maps of the Klein bottle. In addition, we present a simpler construction of a  non-invertible minimal  self-map of  two dimensional torus.  \end{abstract}

\maketitle

\section{Introduction}\label{intro} 
This paper deals with  the minimal  circle-fibered self-maps of two dimensional manifolds.  We recall that 
the pair $(X,f)$ consisting of a compact metric space $X$ and its continuous endomorphism  $f:X \to X$ is called a minimal 
dynamical system if $X$ does not have any non-empty compact subset $X'\not=X$ satisfying $f(X') \subseteq X'$. 
The understanding of the structure of minimal systems has a clear importance for the discrete dynamics. 
During the last decades,  much progress has been made in studying minimal  subsystems of $(M,f)$ in the case when $M$ is a low dimensional compact connected manifold, e.g. see [1--14].

In particular, Auslander and  Katznelson have proved \cite{AK} that the minimality of  $(M,f)$ together with dim $M =1$ implies that $M=S^1$ and that $f$ is conjugate to an irrational rotation (hence, $f$ is a homeomorphism).   
If dim~$M =2$ then, 
due to the  Blokh-Oversteegen-Tymchatyn theorem \cite{B,KST08}, the minimal manifold $M$ must be either the 2-torus $\mathbb{T}^2$  or the Klein bottle $\mathbb{K}^2$.   It was  also shown in \cite{KST} that, in contrast with the minimal system $(S^1,f)$,  there exist minimal fiber-preserving  systems $(\T^2,f)$ which are not invertible.  The key dynamical and topological components of the proof in \cite{KST} are, respectively,  the Rees example \cite{MR}  of a non-distal but point-distal torus homeomorphism and  the Roberts-Steenrod characterisation \cite{RS} of the  monotone transformations of 2-dimensional manifolds.

Since the available constructions \cite{Ell,Parry} 
 of the minimal homemorphisms  of  the Klein bottle are technically  quite involved,  the similar question about the existence of minimal non-invertible self-maps of $\mathbb{K}^2$  has been left open in \cite{B,KST,KST08}. 
In fact, more complicated topology of the Klein bottle (a skew product of two circles) in comparison  to the torus $\mathbb{T}^2$  (a direct product of two circles) could potentially be an obstacle for the existence of minimal non-invertible self-maps of $\mathbb{K}^2$, cf. 
\cite[Theorem C-11 and Corollary 1]{KST13} and \cite{DM,GW}.  Nevertheless, the main result of this paper shows that
\begin{theorem} \label{main} There exists a fiber-preserving transformation $\widetilde{S}$ of the Klein bottle, which is minimal and non-invertible.\end{theorem}
Theorem \ref{main} is proved in the next section of our work. The proof uses the Hric-J\"ager blow up technique proposed  recently in \cite{HrJ}. In the cited work, the authors also sketched  a new construction of a non-distal but point-distal torus homeomorphism.  We develop further their construction and adopt it to a more complicated topological situation. As a by-product, even without the use of 
the Roberts-Steenrod theory of monotone transformations of  2-dimensional manifolds, we are able to  present a relatively short and explicit analytic construction of a fiber-preserving minimal non-invertible self-map of  the two-torus, cf. \cite{KST}.

\section{Proof of the main theorem} \label{AF}
Set $\mathbb{T}^1:=\mathbb{R}/{\mathbb{Z}}$ then $\mathbb{T}^2= \mathbb{T}^1\times\mathbb{T}^1$.  We will fix the positive orientation of $\T^1$ induced by  the usual order on $[0,1)$.
Each ordered pair $(x,y)$ of points in $\T^1$ defines two closed subarcs of $\T^1$ whose endpoints are $x$ and 
$y$.   The arc obtained by moving a point from $x$ to $y$ in the positive direction will  be denoted by 
 $[x,y]  \subset \T^1$.  Hence, $[x,y] \cup [y,x] = \T^1$  so that  $0.5 \in [0.25,0.75]$ and $0\in [0.75,0.25]$. By slightly abusing the  notation, we also will write $[0,1] = \T^1, \ [0,0.5]= [1,0.5]$ and $[0,0] = \{0\}$. 

Next, consider the  homeomorphism $P:\T^2 \to \T^2$ defined by $P(x,y)=(x+1/2, $ $1-y).$
Let $\sim$ be an equivalence relation on $\T^2$ in which each point $(x,y)$ is identified with all its images: $P^0(x,y)= (x,y)$ and $P(x,y)$.  Let $\pi: \T^2\to \mathbb{K}^2$ denote the corresponding quotient map.
The quotient space $\mathbb{K}^2$ is one of standard models of the Klein bottle. Notice that a transformation $Q$ of the torus induces a transformation of the Klein bottle by the quotient map $\pi$ if and only if $Q$ commutes with $P$.

The desired minimal map $\widetilde{S}: \mathbb{K}^2 \to \mathbb{K}^2$ will be constructed as a factor of a minimal and non-invertible transformation $\widehat{S}$ of the torus. On the other hand,  the map $\widehat{S}$ will be constructed as a topological extension of the Parry minimal homeomorphism $S: \T^2 \to \T^2$ of the form 
$S(x,y)=(R(x),\sigma_x(y)):= (x+\alpha, y +r(x))$. 
Here $R(x)$ is a rotation by an irrational angle $\alpha$ and  continuous function $r: \T^1 \to \R$ is 
such that $r(x)=-r(x+1/2)$ (i.e. $S$ commutes with $P$).  
Moreover,  the Fourier coefficients of $r(x)$ must  satisfy several assumptions listed in \cite{Parry};  in addition, we can choose them in such a way that $r(0)=r(1/2) =0,$  and $r(x) \in (0, 1/4)$  for all $x\in(0,1/2)$. In the sequel, we will use 
the notation  $\sigma_x^n$ for the composition  $\sigma_{R^{n-1}(x)}\circ\sigma_{R^{n-2}(x)}\circ\ldots\circ\sigma_{R(x)}\circ\sigma_x.$

Take now some point $x_1^* \in (0.1,0.2)\cap \Q$, set $x_2^* = x_1^*+1/2$ and then choose the points $z_1^\ast=(x_1^\ast,y_1^\ast)$ and $z_2^\ast=Pz_1^*= (x_2^\ast,y_2^\ast)$ in such a way that $y_j^\ast \not= \sigma^{-m}_{R^m(x_j^\ast)}(0),   \sigma^{-m}_{R^m(x_j^\ast)}(-r(R^m(x_j^\ast))))$ 
for each $m \in \Z, \ j =1,2$. Then  the $S$-orbits of $z_1^*, z_2^*$ do not intersect curves $\mathbb{T}^1\times\{0\}$ and $\{(x,-r(x))| x\in\mathbb{T}^1\}$.

Let continuous  $\psi, \phi: \mathbb{T}^1\rightarrow\mathbb{T}^1$ have their graphs $P$-invariant  and intersecting transversally at the points $z_1^\ast$ and $z_2^\ast$, see Figure 1. We will choose $\phi, \psi$ in such a way that they take  zero values  (recall that $1$ is identified with $0$) in all points except for some neighbourhood $(\bar x_1,\bar x_2)$ of $x_1^\ast$ and $(\bar x_1+1/2,\bar x_2+1/2)$ of $x_2^\ast$, respectively. 
\begin{figure}[h]
\centering \fbox{\includegraphics[width=5cm]
{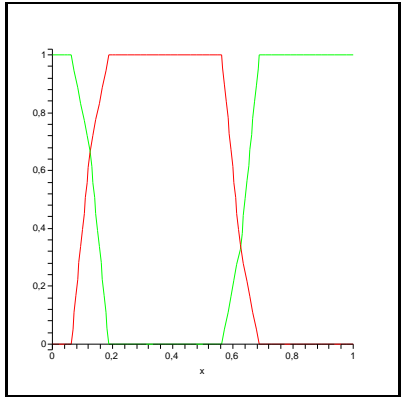}} 
\caption{$P$-invariant graphs of $\phi$ (green) and $\psi$ (red).} \label{FF1}
\end{figure}

 We define a fiber measure $\mu_x^0$ on  the $\sigma$-algebra $\mathcal B$ of Borelian subsets of $\mathbb{T}^1$ as 
$$\mu_x^0:=\left\{\begin{array}{cc} \delta_{y_j^*},&\mbox{if} \ x=x_j^*, \ j =1,2,\\  \frac{\lambda|_{[\psi(x),\phi(x)]}}{\phi(x)-\psi(x)}, & \mbox{if} \  \phi(x)>\psi(x),\\ \frac{\lambda|_{[\phi(x),\psi(x)]}}{\psi(x)-\phi(x)}, & \mbox{if} \  \psi(x)>\phi(x),\end{array}\right.$$
where $\delta_y$ denotes a probabilistic Dirac 
measure concentrated at $y$ (i.e. $\delta_y(y)=1$) and $\lambda|_{[a,b]}(A)$ denotes the Lebesgue measure of intersection of a measured set $A$ and the arc $[a, b]$. 
Consider also the measures $\mu_x^n, \ n \in \Z,$ and $\mu_x$  on $\mathcal B$ defined by
$$\mu_x^n:=\mu_{R^{n}(x)}^0\circ\sigma_x^n, \quad 
\mu_x:=(\lambda+\sum_{n=0}^\infty 2^{-n-1}\mu_x^{n})/{2}.$$
It is clear that $\mu_x(\T^1) =1,$  $x \in \T^1$.  In addition, $\mu$ inherits the symmetry  properties  of $r, \phi, \psi$:
\begin{lemma}
For all $x \in \mathbb{T}^1$ and  $y \in [0,1]$, we have
\begin{equation}\label{MMM}\mu_x[0,y]=1-\mu_{x+\frac12}[0,1-y]. \end{equation}
\end{lemma}
\begin{proof}
First we prove that,  for all $x, y \in \mathbb{T}^1$, it holds 
\begin{equation}\label{MM}\mu_x^0[0,y]=1-\mu_{x+\frac12}^0[0,1-y]. \end{equation}
For $x \not\in E:= [(\bar x_1,\bar x_2) \cup (\bar x_1+1/2,\bar x_2+1/2)]\setminus \{x_1^*, x_2^\ast\}$ 
(i.e. for the fibers with the usual Lebesgue or Dirac measures), this relation is immediate. 
Now, let $x\in E$ be such that $\phi(x)<\psi(x)$ and hence $\psi(x+1/2)=1-\psi(x)<\phi(x+1/2)=1-\phi(x)$. Then
\begin{multline}\nonumber \mu_{x+\frac12}^0[0,1-y]=\left\{\begin{array}{cl}0,& \mbox{if} \ 1-y<\psi\left(x+\frac12\right),\\ \frac{1-y-\psi\left(x+\frac12\right)}{\phi\left(x+\frac12\right)-\psi\left(x+\frac12\right)},& \mbox{if} \  \psi\left(x+\frac12\right)\leq1-y\leq\phi\left(x+\frac12\right),\\1,& \mbox{if} \  1-y>\phi\left(x+\frac12\right);\end{array}\right\}\\=\left\{\begin{array}{cl}0, & \mbox{if} \   y>\psi(x),\\ \frac{\psi(x)-y}{\psi(x)-\phi(x)}, & \mbox{if} \  \phi(x)\leq y\leq \psi(x),\\ 1,& \mbox{if} \ y<\phi(x);\end{array}\right\}= \quad 1-\mu^0_x[0,y].  \end{multline}
Similarly,  (\ref{MM}) holds when $\phi(x)>\psi(x)$. 
Next, we claim that 
\begin{equation} \label{1-a}\mu_x^0[a,b]=1-\mu_{x+\frac12}^0[1-a,1-b], \ \mbox{for} \ a,b \in [0,1],\ a <b,  \ [a,b] \in \T^1. \end{equation}
This is trivial for the fibers over $x_1^\ast$ and $x_2^\ast$. For the rest of the fibers,  (\ref{1-a})  follows from (\ref{MM}) in view of 
$$\mu_x^0[a,b]=\left\{\begin{array}{cl}\mu_x^0[0,b]-\mu_x^0[0,a],& \mbox{if} \ 0\leq a\leq b\leq 1;\\ 1-\mu_x^0[0,a]+\mu_x^0[0,b], & \mbox{if} \ 0\leq b< a\leq 1.\end{array}\right.$$
Finally, we will show that the superscript $0$ in (\ref{MM}) can be omitted. Indeed, set $\varrho_n(x):=r(x)+r(x+\alpha)+r(x+2\alpha)+\ldots+r(x+(n-1)\alpha)$. Then 
\begin{multline}\nonumber \mu_{x+\frac12}[0,1-y]=\sum_{n \geq 0} 2^{-n-2}\mu_{R^{n}(x+\frac12)}^0\sigma_{x+\frac12}^{n}[0,1-y]+(1-y)/2=\\
\sum_{n \geq 0} 2^{-n-2}\mu^0_{x+\frac12+n\alpha}\left[\varrho_n\left(x+\frac12\right),1-y+\varrho_n\left(x+\frac12\right)\right]+(1-y)/2=\\
\sum_{n \geq 0} 2^{-n-2}\mu_{x+n\alpha+\frac12}^0[-\varrho_n(x),1-y-\varrho_n(x)]+1/2-y/2=\\
\sum_{n \geq 0} 2^{-n-2}(1-\mu_{x+n\alpha}^0[1+\varrho_n(x),y+\varrho_n(x)])+1/2-y/2=\end{multline}
\hspace{1cm}$
\sum_{n \geq 0} 2^{-n-2}+1/2-\left(\sum_{n \geq 0} 2^{-n-2}\mu_{R^{n}(x)}^0\sigma_{x}^{n}[1,y]+y/2\right)
=1-\mu_x[0,y].$
\end{proof}
The next two results show that  $\mu_x^n[0,y]$ and $\mu_x[0,y]$ have discontinuities only at the points from backward orbits of $(x_1^\ast, y_1^*)$ and $(x_2^\ast,y_2^*)$:
\begin{lemma} \label{22} Let $\{x_j\}, \{y_j\}, \{z_j\}$ be sequences in $\mathbb{T}^1$ which converge to $x_0$ and $y_0, z_0$, respectively. Then
$$\limsup_{j\rightarrow\infty}\mu_{x_j}^0[y_j,z_j]\leq\mu_{x_0}^0[y_0,z_0].$$
Moreover, if $(x_0,y_0) \not= (x_k^*,y_k^*), k =1,2$, then 
$$\liminf_{j\rightarrow\infty}\mu_{x_j}^0[y_j,z_j]\geq\mu_{x_0}^0[y_0,z_0).$$
\end{lemma}
\begin{proof} Since the functions $\phi, \psi$ are continuous, this result is a straightforward consequence of the definition of  $\mu_x$. In fact, for each $x_0 \not=x_k^*$, we have  that
$\mu_{x_j}^0[y_j,z_j]\to\mu_{x_0}^0[y_0,z_0]$ as $j \to +\infty$. 
\end{proof}
\begin{corollary}\label{pi3} Let $\{x_j\}$ and $\{y_j\}$  converge to $x_0$ and $y_0$, respectively. Then
\begin{equation}\label{pi}
\mu_{x_0}^n[0,y_0) \leq \liminf_{j\rightarrow\infty}\mu_{x_j}^n[0,y_j]\leq \limsup_{j\rightarrow\infty}\mu_{x_j}^n[0,y_j]\leq\mu_{x_0}^n[0,y_0], \quad n \in \N.
\end{equation}
Furthermore,
$$\mu_{x_0}[0,y_0) \leq \liminf_{j\rightarrow\infty}\mu_{x_j}[0,y_j]\leq \limsup_{j\rightarrow\infty}\mu_{x_j}[0,y_j]\leq\mu_{x_0}[0,y_0].$$
\end{corollary}
\begin{proof} Since $\sigma, \ R$ are continuous functions and,  by our assumption, $(R^n(x), \sigma^n_x(0))\not=(x_k^*,y_k^*), k =1,2,$ for each $x \in \mathbb{T}^1, \ n \in \N$,  we have that $ \mu_{x_0}^n[0,y_0)=$
$$
= \mu_{R^n(x_0)}^0[\sigma^n_{x_0}(0),\sigma^n_{x_0}(y_0))  \leq \liminf_{j\rightarrow\infty}\mu_{R^n(x_j)}^0[\sigma^n_{x_j}(0),\sigma^n_{x_j}(y_j)]  \leq \liminf_{j\rightarrow\infty}\mu_{x_j}^n[0,y_j].$$  The proof of the second inequality in (\ref{pi}) is similar.  Finally,  by a direct calculation we get the following: $$\limsup_{j\rightarrow\infty}\left(\sum_{n\geq0}2^{-n-2}\mu_{x_j}^n[0,y_j]+y_j/2\right)\leq\sum_{n\geq0}2^{-n-2}\mu_{x_0}^n[0,y_0]+y_0/2=\mu_{x_0}[0,y_0],$$
$$\liminf_{j\rightarrow\infty}\left(\sum_{n\geq0}2^{-n-2}\mu_{x_j}^n[0,y_j]+y_j/2\right)\geq\sum_{n\geq0}2^{-n-2}\mu_{x_0}^n[0,y_0)+y_0/2=\mu_{x_0}[0,y_0). $$
Hence, if  $\mu_{x}[0,y]$ is discontinuous at some point $(x_0,y_0)$ then $\mu_{x_0}\{y_0\} >0$.
\end{proof}
Following \cite{HrJ}, we consider continuous fiber-preserving  selfmap  $T:\ \mathbb{T}^2\rightarrow\mathbb{T}^2$ defined by  $T(x,y):=(x,\tau_x(y))$, where 
$$\tau_x(y):=\min\{y^\prime\in[0,1]|\mu_x[0,y^\prime]\geq y\}.$$
The existence of this  minimum can be deduced, for example, from Corollary \ref{pi3}.
It is clear that  $\tau_x(0)=0$ and $\tau_x(1)=1$ because, for all small $\epsilon >0$, 
 $$\mu_x[0,1-\epsilon]=1-\mu_{x+\frac12}[0,\epsilon]=1-\sum_{n \geq 0} 2^{-n-2}\mu_{x+\frac12}^n[0,\epsilon]-\epsilon/2 <1.$$ 
Obviously,  $\tau_x(y)$ is an increasing function of $y\in (0,1)$. 
Now, let $\mathrm{Orb}_R(x) = \{R^j(x),\ j \in \Z\}$ [respectively, $\mathrm{Orb}^-_R(x) = \{R^j(x),\ j \leq 0\}$]
denote the full [respectively, backward] orbit of a point $x \in \mathbb{T}^1$. As we have proved, $\mu_x[0,y]$ is continuous at each point 
$(x,y)$ where $x \not\in  \frak{O^-}:= \mathrm{Orb}^-_R(x_1^\ast)\cup\mathrm{Orb}^-_R(x_2^\ast)$.   Therefore 
\begin{equation}\label{dt}
\mu_x[0,\tau_x(y)]= y \ \mbox{if} \ x \not\in \frak{O}^-
\end{equation}
that implies $\tau_x(y_1) < \tau_x(y_2)$ for all $0< y_1<y_2 \leq 1$ and $x \in \T^1\setminus \frak{O}^-$. 

\begin{lemma} \label{lmmContT} The map $T: \mathbb{T}^2\rightarrow\mathbb{T}^2$ is continuous and surjective. Moreover, $T$  commutes with $P$ and is invertible on the set $\mathbb{T}^2\setminus(\frak{O}^-\times\mathbb{T}^1)$. \end{lemma}
\begin{proof} It is clear that $T$ is continuous if and only if  $\tau_x(y)$ is a continuous function  of $x, y$.
So, take some $(x_0, y_0)$ and consider sequences $x_j \to x_0$ and $y_j \to y_0$.  From the definition of $\tau$ it holds $\mu_{x_j}[0,\tau_{x_j}(y_j)]\geq y_j$.
Suppose that $\tau_{x_j}(y_j)$ converges to some limit point $z$. Then Corollary \ref{pi3}  yields 
$\tau_{x_0}(y_0)\leq z = \lim_{j\to\infty}\tau_{x_j}(y_j).$ 
Suppose for a moment that $\tau_{x_0}(y_0) < z$ and take $\delta >0$ such that $\tau_{x_0}(y_0)+\delta < \tau_{x_j}(y_j)$ for all large $j$. 
Then 
$$\mu_{x_0}[0,\tau_{x_0}(y_0)+\delta)=\sum_{n\geq0}2^{-n-2}\mu_{x_0}^n[0,\tau_{x_0}(y_0)+\delta)+(\tau_{x_0}(y_0)+\delta)/2 \geq y_0+\delta/2$$
and therefore,  due to  Corollary \ref{pi3}, we have for all large $j$ that
$$y_j < y_0+0.25\delta < \mu_{x_0}[0,\tau_{x_0}(y_0)+\delta) \leq \liminf_{j\rightarrow\infty}\mu_{x_j}[0, \tau_{x_0}(y_0)+\delta]. $$ 
As a consequence, $\tau_{x_{j}}(y_{j}) \leq \tau_{x_0}(y_0)+\delta$ for all large $j$, a contradiction. Hence, $\tau_{x_0}(y_0)= \lim_{j\to\infty}\tau_{x_j}(y_j)$ and  $\tau_x(y)$ is continuous.  

Now, since $\tau_x(y)$ depends continuously on $x,y$ and $\tau_x(0)=0$ and $\tau_x(1)=1$, we obtain that $\tau_x(\T^1) = \T^1$.  In addition, if $x \not\in\frak{O}^-$ then $\tau_x(y)$ is a strictly increasing function of $y$ and therefore $T$   is invertible on  $\mathbb{T}^2\setminus(\frak{O}^-\times\mathbb{T}^1)$.

Next, since  
$$(T\circ P)(x,y)=\left(x+\frac12,\tau_{x+\frac12}(1-y)\right), \ (P\circ T)(x,y) = \left(x+\frac12,1-\tau_x(y)\right),$$
we find that $T$ commutes with $P$ if and only if
 $$\tau_{x+\frac12}(1-y)=1-\tau_x(y) \ \mbox{for all} \ x,y. $$
Now, for $x\in\mathbb{T}^1\setminus \frak{O}^-$ it holds that
$$1=\mu_x[0,\tau_x(y)]+\mu_{x+\frac12}[0,1-\tau_x(y)]=y+\mu_{x+\frac12}[0,1-\tau_x(y)],$$
so  that
$$\mu_{x+\frac12}[0,1-\tau_x(y)]=1-y=\mu_{x+\frac12}[0,\tau_{x+\frac12}(1-y)].$$
Thus $T$ commutes with $P$ on $\mathbb{T}^2\setminus(\frak{O}^-\times\mathbb{T}^1)$.  Finally, if $(\hat x, \hat y)\not\in \mathbb{T}^2\setminus(\frak{O}^-\times\mathbb{T}^1)$ we can find a sequence of points $(x_j,y_j)\in \mathbb{T}^2\setminus(\frak{O}^-\times\mathbb{T}^1)$ such that $(x_j,y_j) \to (\hat x, \hat y)$ as $j \to +\infty$. But then 
$$T\circ P(\hat x,\hat y)=\lim_{j \to +\infty}T\circ P(x_j,y_j)=\lim_{j \to +\infty}P\circ T(x_j,y_j)=P\circ T(\hat x, \hat y).$$
This completes the proof. 
\end{proof}
We are ready to construct the non-invertible minimal map $\widehat S:\T^2 \to \T^2$.  Defining $\widehat S$ on the set\, $\Lambda:=\mathbb{T}^2\setminus(\frak{O}\times\mathbb{T}^1)$, where $\frak{O}:= \mathrm{Orb}_R(x_1^\ast)\cup\mathrm{Orb}_R(x_2^\ast)$,  by
$$\widehat S|_\Lambda:=T^{-1}|_\Lambda\circ S|_\Lambda\circ T|_\Lambda,$$ we will extend it continuously on the whole torus. 
\begin{lemma} For each $n\in \Z$, the map $\theta_n(x):= \mu_{x}^n[-r(x),0]$ \mbox{is continuous on $\T^1$}. \label{lmmpom}
\end{lemma}
\begin{proof} We observe that $\theta_n(0) =\theta_n(1/2) = \theta_n(1)=0$ 
and thus it suffices to establish the continuity 
of $\theta_n$ on the arcs $[0,1/2]$ and $[1/2,1]$ of $\T^1$ separately. For instance, consider the arc $[0,1/2]$ where 
$-r(x)\leq 0$ so that $\theta_n(x) = 1 - \mu^n_x[0, 1-r(x)]$. We recall  that, by our assumptions, none of the points $(R^n(x),\sigma_x^n(1-r(x)))= (R^n(x), \sigma_x^n(-r(x))), \ n \in \Z$, coincided with  $(x_1^\ast, y_1^*)$ and $(x_2^\ast,y_2^*)$ and therefore the function $\mu^n_x[0,y]$ is continuous at each point of the form $(x,y) = (x, 1-r(x))$.  In consequence, 
the map $x  \to \theta_n(x), \ x \in [0,1/2]$ is continuous  as composition of two continuous applications: $x \to (x, 1-r(x)), x \in[0,1/2],$ and $(x,y) \to  1 - \mu^n_x[0, y],$ $(x,y) \in \{(x, 1-r(x)):\ x \in [0,1/2]\}$.  
\end{proof}

Next, set\, $\Lambda_-:=\Lambda\cap\left([0,1/2]\times\mathbb{T}^1\right)$ and\, $\Lambda_+:=\Lambda\cap\left([1/2,1]\times\mathbb{T}^1\right)$.

\begin{lemma} The map $\widehat S$ is  uniformly continuous on $\Lambda_-$. \end{lemma}
\begin{proof}
It follows from (\ref{dt}) that  
$\tau_x^{-1}(y)=\mu_x[0,y]$ for each pair $(x,y) \in\Lambda. $ Hence, 
$$\widehat S(x,y)=(T^{-1}\circ S\circ T)(x,y)=(R(x),\mu_{R(x)}[0,\sigma_x(\tau_x(y))]), \ (x,y) \in\Lambda.$$
Since $R:\T^1 \to \T^1$  is continuous, we need only to prove the uniform continuity of 
$M(x,y):=\mu_{R(x)}[0,\sigma_x(\tau_x(y))]$ on\, $\Lambda_-$. This task can be simplified if we observe  that, due to the Weierstrass M-test and  Lemmas \ref{lmmContT} and \ref{lmmpom}, the function 
$W(x,y):= \sum_{n\geq 0} 2^{-n-2}\mu^{n+1}_{x}[-r(x),0]+  \sigma_x(\tau_x(y))/2$ is continuous on $\T^2$ while $$
M(x,y)= \sum_{n\geq 0} \frac{\mu^n_{R(x)}[0,\sigma_x(\tau_x(y))]}{2^{n+2}} +  \sigma_x(\tau_x(y))/2 = \sum_{n\geq 0} \frac{\mu^{n+1}_{x}[0,\tau_x(y)]}{2^{n+2}} +W(x,y). 
$$
Here we are using the  relation 
\begin{multline} \nonumber \mu^n_{R(x)}[0,\sigma_x(\tau_x(y))] =\mu_{R^{n}(R(x))}^0\circ\sigma_{R^{n-1}(R(x))}\circ\ldots\circ\sigma_{R(R(x))}\circ\sigma_{R(x)}[0,\sigma_x(\tau_x(y))]=\\
{}=\mu_{R^{n+1}(x)}^0\circ\sigma_{R^n(x)}\circ\ldots\circ\sigma_{R(x)}\circ\sigma_x[\sigma_x^{-1}(0),\tau_x(y)]=\\
=\mu_x^{n+1}[-r(x),\tau_x(y)]=\mu_x^{n+1}[-r(x),0]+\mu_x^{n+1}[0,\tau_x(y)].\end{multline}
Recall also that $\sigma_x(y)=y+r(x), \ \mu_x^{n}\{0\}=0, \ n \in \N$,  and  that $0\in[-r(x),\tau_x(y)]$ because of non-negativity of  $r(x)$ for $x\in[0,1/2]$. 

In consequence, it suffices to establish the uniform continuity of the function 
$A(x,y):=\sum_{n\geq 0} 2^{-n-2}\mu^{n+1}_{x}[0,\tau_x(y)] $ on\, $\Lambda_-$. In other words, we have to prove that 
for each $\epsilon>0$ there exists $\delta>0$ such that for any $x_1,x_2\in[0,1/2]\setminus\frak{O}$ and $y_1,y_2\in\mathbb{T}^1$ satisfying
$d(x_1,x_2)<\delta \mbox{ and } d(y_1,y_2)<\delta$
it holds that
$\frak{a}:= d(A(x_1,y_1), A(x_2,y_2)) <\epsilon.$ Here $d$ denotes the metric on $\T^1$ naturally inherited from $\R$. In particular, $d$ is shift-invariant. 
Observe also that we interpret 
$A(x,y)$ as a point on $\T^1$.  
Set $s_n:=\mu_{x_1}^n[0,\tau_{x_1}(y_1)]-\mu_{x_2}^n[0,\tau_{x_2}(y_2)].$ Since
 $y_k= \mu_{x_k}[0,\tau_{x_k}(y_k)]$, we find that 
$$
\frak{b}:=d(y_1,y_2)=d\left(\sum_{n\geq 0} \frac{s_{n}}{2^{n+2}} + (\tau_{x_1}(y_1)-\tau_{x_2}(y_2))/2,\ 0\right),
\frak{a}=d\left(\sum_{n\geq 0} \frac{s_{n+1}}{2^{n+2}},\ 0\right).
$$
So let us take arbitrary $\epsilon>0$, then there exists $N_1$ for which
$$ \sum_{n\geq N_1} {2^{-n-2}}<\frac\epsilon{16}, \ \ \mbox{so that } \ \sum_{n\geq N_1}d({2^{-n-2}}s_{n+1}, 0)<\frac\epsilon{16}.$$
Let  $\mathcal{U}^n_\kappa$ denote  open $\kappa-$ neighbourhood of the set $\{R^{-n}(x^*_1), R^{-n}(x^*_2)\}$.   Due to  Corollary \ref{pi3} and Lemma \ref{lmmContT}, the function \mbox{$\mu_{x}^n[0,\tau_{x}(y)]: (\T^1\setminus \{\mathcal{U}^n_\kappa\})\times \T^1  \to \T^1$}  is uniformly continuous for each $n$.  We can choose $\kappa$ to be small enough in order 
to have the closures of $N_1+1$  sets $\mathcal{U}^n_{2\kappa}, \ n =0,1, \dots, N_1,$ mutually disjoint.  Obviously, each function from the finite set $\{\tau_{x}(y)$, $\mu^{n}_{x}[0,\tau_x(y)], n =0,1, \dots, N_1\}$,  is uniformly continuous on 
$(\T^1\setminus \{\cup_{n=0}^{N_1}\mathcal{U}^n_{\kappa}\})\times \T^1$.  

Take now $\delta\in (0, \min\{\epsilon/16, \kappa\})$  small enough to assure
 that $d(x_1,x_2)<\delta,$  $d(y_1,y_2)<\delta$ imply the inequality $$
d(\tau_{x_2}(y_2),\tau_{x_1}(y_1)) < \epsilon/8
$$
as well as 

(i) the existence of at most one integer $n_0\in [0, N_1]$ such that $\{x_1, x_2\} \cap \mathcal{U}^{n_0}_{2\kappa}\not=\emptyset$;

(ii) $d({2^{-n-2}}s_n,0) < \epsilon/(16(N_1+1))$ once $\{x_1, x_2\} \cap \cup_{n=0}^{N_1}\mathcal{U}^n_{\kappa}=\emptyset$, $n=0,1,\dots, N_1$. 

A key observation is that  in the case (i) the distance $d({2^{-n_0-2}}s_{n_0},0)$ cannot be large even when $\{x_1, x_2\} \cap \cup_{n=0}^{N_1}\mathcal{U}^n_{\kappa}\not=\emptyset$: 
\begin{multline}\nonumber d({2^{-n_0-2}}s_{n_0},0)\leq
\frak{b} + d\left(\sum_{n=0, n \neq n_0}^{ N_1}{2^{-n-2}}s_{n},0\right) + d\left(\sum_{n>N_1} {2^{-n-2}}s_n,0\right) +{}\\
{}+d(\tau_{x_1}(y_1)-\tau_{x_2}(y_2),0)/2
< \frac{\epsilon}{16} + (N_1+1)\frac\epsilon{16(N_1+1)}+\frac\epsilon{16}+\frac{\epsilon}{16}=\frac\epsilon{4}.\end{multline}
Thus, estimating separately the term  $d({2^{-n_0-1}}s_{n_0},0)$ (whenever $n_0$ with properties described in (i) appears) as $d({2^{-n_0-1}}s_{n_0},0) <\epsilon/2$, we obtain 
$$ \frak{a} = d\left(\sum_{n=0}^\infty \frac{s_{n+1}}{2^{n+2}},\ 0\right) \leq  \frac\epsilon{2} + d\left(\sum_{n < N_1, n\neq n_0}\frac{s_{n+1}}{2^{n+2}},0\right)
+d\left(\sum_{n\geq N_1}\frac{s_{n+1}}{2^{n+2}},0\right)< 
$$ 
$
\epsilon/2+ 2N_1\epsilon/(16(N_1+1))+\epsilon/16< \epsilon$, which completes the proof. 
\end{proof}
\begin{corollary} The map $\widehat S|_{\Lambda}$  commutes with $P$,  is  uniformly continuous on $\Lambda$ and it admits a unique continuous extension $\widehat S$ on $\T^2$ which also commutes with $P$. 
\end{corollary}
\begin{proof} Observe that all maps $T, P, S, \widehat S: \Lambda \to \Lambda$\, are bijective and $P({\Lambda_+}) =  ({\Lambda_-})$. First, we will prove that  $\widehat S|_\Lambda$ commutes with $P$.  Clearly, 
$$(S\circ T\circ P)(z) = (S\circ P\circ T)(z) =  (P\circ S\circ T)(z), \ z \in \T^2, $$
so that, for $z \in \Lambda$, 
$$(T\circ \widehat S\circ P)(z) = (S\circ T\circ P)(z) = (P\circ S\circ T)(z) = (P\circ T\circ \widehat S)(z) = (T\circ P\circ \widehat S)(z).$$
But $T$ is injective on\, $\Lambda$ \, and therefore 
$(P\circ \widehat S)(z)=(\widehat S\circ P)(z)$ for each $z \in \Lambda$. Hence, since $P, P^{-1}$ are linear maps, 
$ \widehat S|_{\Lambda_+} = P^{-1}\circ \widehat S|_{\Lambda_-}\circ P|_{\Lambda_+}$ is also uniformly continuous.
As the maps $\widehat S|_{\Lambda^-}$ and $\widehat S|_{\Lambda^+}$ are uniformly continuous and\, $\Lambda^-$ is dense in $[0,1/2]\times\mathbb{T}^1$ and $\Lambda^+$ is dense in $[1/2,0]\times\mathbb{T}^1$, they can be uniquely continuously extended to the sets $[0,1/2]\times\mathbb{T}^1$ or $[1/2,0]\times\mathbb{T}^1$, respectively. Since these maps coincide on the intersection $\{0, 1/2\} \times \T^1 = ([0,1/2]\times\mathbb{T}^1) \cap ([1/2,0]\times\mathbb{T}^1)$,
 they define a continuous self-map $\widehat S:\T^2\to \T^2$. Clearly, since $\widehat S|_{\Lambda}$  commutes with $P$, the set\, $\Lambda$\, is dense in $\T^2$ and the functions $\widehat S, P$ are continuous on $T^2$, we obtain that  $(P\circ \widehat S)(z)=(\widehat S\circ P)(z)$ for all $z \in \T^2$.  Similarly, $T\circ \widehat S=S\circ T$ on $\T^2$.
\end{proof}
\begin{lemma} The map $\widehat S:\T^2 \to \T^2$ is minimal and non-invertible.
\end{lemma}
\begin{proof}    
Let $F\subset\mathbb{T}^2$ be a non-empty compact and $\widehat S$-invariant set. But then
$$S\circ T(F)=T\circ \widehat S(F)\subseteq T(F).$$
So, $T(F)$ is also compact $S$-invariant set, which means $T(F)=\mathbb{T}^2$ because $S$ is minimal. But since fibers are mapped to fibers and $T$ is bijective on\, $\Lambda$, the set $F$ must contain whole $\Lambda$. But then $\bar F\supseteq\bar\Lambda=\mathbb{T}^2$ and therefore  $\widehat S$ is a minimal map.

Finally, we prove that $\widehat S$ is non-invertibile. For $\delta \in (0, 0.25)$, consider the points  $y_1=\mu_{x_1^\ast}[0,y_1^\ast] > y_2=\mu_{x_1^\ast}[0,y_1^\ast]-\delta$ on the circle $\mathbb{T}^1$.  For every $y' < y_1^*$, we have that 
\begin{multline}\nonumber \mu_{x_1^\ast}[0,y^\prime]=\sum_{n=0}^\infty \frac{\mu_{R^n(x_1^\ast)}^0\sigma_x^n[0,y^\prime]}{2^{n+2}}+y^\prime/2=\sum_{n=1}^\infty \frac{\mu_{R^n(x_1^\ast)}^0\sigma_x^n[0,y^\prime]}{2^{n+2}}+y^\prime/2<\\
<\sum_{n=1}^\infty \frac{\mu_{R^n(x_1^\ast)}^0\sigma_x^n[0,y^*_1]}{2^{n+2}}+y^*_1/2+0.25-\delta= \sum_{n=0}^\infty  \frac{\mu_{R^n(x_1^\ast)}^0\sigma_x^n[0,y^*_1]}{2^{n+2}}+y_1^\ast/2-\delta= y_2.\end{multline}
This yields immediately that $\tau_{x_1^\ast}(y_1)=y_1^\ast=\tau_{x_1^\ast}(y_2)$
and therefore 
$$T\circ \widehat S(x_1^\ast,y_1)=  S(x_1^\ast,y_1^*) = (R(x_1^*),\sigma_{x_1^*}(y_1^*))= T\circ \widehat S(x_1^\ast,y_2). $$ Since $T$ is invertible on the fiber over $\{R(x_1^*)\}$, we find that  $\widehat S(x_1^\ast,y_1)=   \widehat S(x_1^\ast,y_2)$  and therefore the non-degenerated interval $\{x_1^*\}\times [y_1-0.25,y_1]$ is transformed by 
$\widehat S$ into a point.  \end{proof}

\begin{proof}[Proof of Theorem \ref{main}] So far we have obtained a non-invertible minimal self-map $\widehat S$ of the torus. Since $\widehat S$  commutes with $P$,  it induces a transformation $\widetilde S$ of the Klein bottle. $\widetilde S$ is a factor of $\widehat S$ by the fiber-preserving quotient map $\pi$, so $\widetilde S$ is a non-invertible minimal fiber-preserving transformation $\widehat S$ of the Klein bottle. The construction is completed.
\end{proof}
\begin{remark} After changing the definition of $\mu_x$ (where an appropriate weighted sum of {\it all measures} $\mu^n_x, \ n \in \Z,$ should be considered),  we  can  similarly construct a  minimal circle-fibered homeomorphism of the Klein bottle having an asymptotic pair of points.  Then the Roberts-Steenrod theory  \cite{RS} of monotone transformations of 2-dimensional manifolds can applied in order to obtain a different proof of 
Theorem \ref{main}, see \cite{KST} for more detail. 
\end{remark}

\section{Acknowledgements}
\noindent The authors  thank  {\mL}ubom\'ir Snoha and  Roman Hric  for valuable conversations.  
The first author was supported by Project SGS 2/2013 from the Silesian University in Opava. Support of this institution is gratefully acknowledged. The second author was partially supported  by FONDECYT (Chile), project 1110309.  This work was done during the second author sabbatical leave from the University of Talca and his research stay  at the Mathematical Institute of the Silesian University in Opava.  Sergei Trofimchuk  acknowledges the hospitality of this Mathematical Institute.

\bibliographystyle{amsplain}

\end{document}